\documentclass{amsart}
\usepackage{amsmath,amsthm,amssymb,amsfonts,amscd, array}
\usepackage[all]{xy}

\newtheorem{thm}{Theorem}[section]

 \newtheorem{defn}[thm]{Definition}

\title{General contractions in new type perturbed metric spaces}
\author{Bekir Dan{\i}\c{s}}
\address{
{\bf Bekir Dan{\i}\c{s}} (Corresponding Author) \\
Digital Transformation Office\\
Aydın Adnan Menderes University \\
Central Campus, 09010, Aydın, Turkey \\
}
\email{bekir.danis@adu.edu.tr}
\begin{document}

\begin{abstract}

 We focus on the new type perturbed metric spaces and introduce a contraction mapping namely new type perturbed Kannan mappings. For these mappings, we show that Banach's fixed point theorem holds. Moreover, this new generalization of Banach's contraction principle does not depend on the continuity of the operator. 

 {\bf\subjclassname} { 54E50, 47H09, 47H10}\\
 {\bf Keywords}: fixed point, Kannan mappings, new type perturbed metric spaces.
 \end{abstract}

\maketitle

\section{Introduction and Preliminaries}
\label{Sec:1}
Banach's fixed pont theorem(see \cite{ban}) is an important argument in nonlinear analysis and so there are several established generalizations. For these extensions, we have two main techniques: one is construct a new contraction mapping and the other one is introducing new metric spaces. In this present paper, we study more general contraction mappings for the new type perturbed metric spaces introduced in \cite{b}. For further investigations, we refer the reader \cite{js},\cite{k2},\cite{k1},\cite{cha},\cite{rus}
 
We recall the definition of perturbed metric spaces introduced by Jleli and Samet in \cite{js2}
\begin{defn}
Let $X$ be an arbitrary non-empty set. Assume two mappings are given as follows: $D:  X \times X \rightarrow {[0, \infty })$ and $P:  X \times X \rightarrow {[0, \infty })$. We call $D$ is a perturbed metric on $X$ with respesct to $P$ if $${D- P}:  X \times X \rightarrow {[0, \infty }),$$ $${(D - P)}(x,y)={D(x,y) - P(x,y)}$$ is a metric on $X$.

We say $P$ is perturbed  mapping, $d={D - P}$ is the exact metric and the notation $(X,D,P)$ denotes the perturbed metric space.
\end{defn}
 
Furthermore, the article \cite{js2} includes the extension of Banach's contraction principle to perturbed metric spaces. After this generalization, we introduce a new type perturbed metric spaces in \cite{b} as follows:

\begin{defn}
Let $X$ be an arbitrary non-empty set. Assume two mappings are given as follows: $D:  X \times X \rightarrow {[0, \infty })$ and $P:  X \times X \rightarrow {[c, \infty })$ where $c$ is a positive real number. We call $D$ is a new type perturbed metric on $X$ with respesct to $P$ if $${D \over P}:  X \times X \rightarrow {[0, \infty }),$$ $${D \over P}(x,y)={D(x,y) \over P(x,y)}$$ is a metric on $X$.

We say $P$ is a new type perturbed mapping, $d={D\over P}$ is the exact metric and the notation $(X,D,P)$ denotes the new type perturbed metric space.
\end{defn}

Furthermore, we remind reader of the topological definitions in \cite{b}.

\begin{defn}
Let $(X,D,P)$ be a new type perturbed metric space, $T: X \rightarrow X$ and $\{x_n\}$ denotes a sequence in $X$.
\begin{itemize}
\item[{(a)}] We call  $\{x_n\}$ is a new type perturbed convergent sequence in new type perturbed metric space $(X,D,P)$ if the sequence  $\{x_n\}$ is convergent in $(X,d)$.

\item[{(b)}] We call  $\{x_n\}$ is a new type perturbed Cauchy sequence in $(X,D,P)$ if it is a Cauchy sequence with respect to the exact metric.

\item[{(c)}] We say $(X,D,P)$ is complete new type perturbed metric space if $X$ is complete metric space with respect to the exact metric $d$.

\item[{(d)}] We call $T$ is a new type perturbed continuous mapping if it is a continuous map with respect to $d$.
\end{itemize}
\end{defn}

In the article \cite{b}, it is shown that 

\begin{thm}
Assume $(X,D,P)$ is a complete new type perturbed metric space and $T$ represents a given new type perturbed continuous map from $X$ to itself, that is,

$$D(Tx,Ty) \leq \alpha D(x,y) \text{ for } \alpha \in (0,1).$$
 We proved that $T$ admits a unique fixed point.

\end{thm}

Notice that this generalization is related to change the metric spaces. Whereas, Kannan extends the fixed point theorem by defining general contractions in \cite{kan} and points out that the operator is not necessarily continuous. 
\begin{thm}
Let $(X,d)$ be a complete metric space and $T: X \to X$ is a mapping satisfying 

$$d(Tx,Ty) \leq \alpha [d(x,Tx)+d(y,Ty)]$$ for all $x,y\in X$ where $0\leq \alpha < {1 \over 2}$. Then, Kannan verified that $T$ has a one and only one fixed point.
\end{thm}

After the statement of Kannan's theorem, we investigate Kannan mappings for the new type perturbed metric spaces.
\section{Main Results}
\label{Sec:2}
Throughout this paper,  $(X,D,P)$ is a new type perturbed metric space.
\begin{defn}
 $T: X \to X$ denotes a mapping such that for all $x,y\in X$, 

 \begin{equation} \label{eq1}D(Tx,Ty) \leq \alpha [D(x,Tx)+D(y,Ty)] \end{equation} where $0\leq \alpha < {1 \over 2}$. We say $T$ is a new type perturbed Kannan mapping.
\end{defn}

\begin{thm}
Suppose that $(X,D,P)$ is a complete new type perturbed metric space and $T$ is  a new type perturbed Kannan mapping. Then, there exists a unique fixed point of $T$.
\end{thm}

\begin{proof}
Take an arbitrary element $x_0 \in X$ and consider Picard sequence defined by $$x_{n+1}=T x_n.$$ In the inequality (\ref{eq1}),by taking $x=x_0, y=x_1$, we have $$D(x_1,x_2)=D(Tx_0,Tx_1)\leq \alpha [D(x_0,Tx_0)+D(x_1,Tx_1)]= \alpha [D(x_0,x_1)+D(x_1,x_2)].$$
This means that 

 \begin{equation}\label{eq2} D(x_1,x_2) \leq {{\alpha\over {1-\alpha}} D(x_0,x_1)}. \end{equation}

In a similar way, we get 

$$ D(x_n,x_{n+1}) \leq {{\alpha\over {1-\alpha}} D(x_{n-1},x_n)}.$$ 
For notational convenience, we call $\beta={\alpha\over {1-\alpha}}$ and $D_0=D(x_0,x_1)$. Observe that $\beta \in [0,1)$. Then, we obtain iteratively 

 \begin{equation}\label{eq2}   D(x_n,x_{n+1}) \leq \beta^n D_0. \end{equation}

Recall that $d={D\over P}$ is the exact metric and the image of $P$ is  ${[c, \infty })$. Since $c >0$ and by (\ref{eq2}), it follows that $$d(x_n,x_{n+1})={D(x_n,x_{n+1}) \over P(x_n,x_{n+1})} \leq {\beta^n \over c} D_0.$$

From this inequality, we infer that $\{x_n\}$ is Cauchy sequence in $(X,d)$ by focusing the following steps

\begin{align*} 
d(x_n,x_{n+k}) &\leq {\beta^n \over c} D_0 + \cdots +  {\beta^{n+k-1} \over c} D_0\\
   &=  {\beta^n \over c} D_0 (1+ \cdots + \beta^{k-1} )\\
   &=  {\beta^n \over c} D_0  {({{1-\beta^k} \over {1-\beta}})}\\
   &\leq  {\beta^n \over {(1-\beta)c}} D_0 .
\end{align*}
Notice that $\beta$ is a non-negative real number less than $1$ by remembering $\alpha\in [0,{1\over2})$

Thus, $\{x_n\}$ is  a new type perturbed Cauchy sequence in $(X,D,P)$. By completeness of $(X,D,P)$, there exists  $x^*$ such that \begin{equation}\label{eq4} \lim_{n \to \infty} d(x_n,x^*)=0. \end{equation}

Firstly, we prove $x^*$ is a fixed point of $T$. Apply the condition (\ref{eq1}) by recalling the definition of Picard sequence and we know 

$$D(Tx_n,Tx^*)\leq \alpha [D(x_n,x_{n+1})+D(x^*,Tx^*)].$$

Now, take the limit as n approaches to $\infty$ and get

$$D(x^*,Tx^*)\leq \alpha D(x^*,Tx^*)$$

 by the inequality (\ref{eq2}). 

This inequality indicates that $D(x^*,Tx^*)$ should be $0$ otherwise we have $1\leq\alpha$ and it causes a contradiction with the assumption of $\alpha< {1\over2}$. This implies $d(x^*,Tx^*)=0$ which means $Tx^*=x^*$ and it completes the proof of existence part.

To show the uniqueness part, suppose that $x,y \in X$ are two different fixed points of $T$ and seek a contradiction. Now, we  use the definition of fixed point and the inequality given in (\ref{eq1})

$$D(x,y)=D(Tx,Ty) \leq \alpha [D(x,Tx)+D(y,Ty)=\alpha [D(x,x)+D(y,y)]=0. $$

By previous step, we obtain $d(x,y)\leq0$ and this demonstrates $x=y$. This means that we could not have two distinct fixed points and it finishes the proof.

\end{proof}

\end{document}